\documentclass[11pt]{amsart}
\usepackage[T1]{fontenc}
\usepackage{lmodern,amsmath,amssymb,mathtools}
\usepackage[a4paper,margin=29mm]{geometry}
\usepackage{microtype}
\usepackage{cite}
\usepackage[colorlinks=true,linkcolor=blue,citecolor=blue,urlcolor=blue]{hyperref}
\hypersetup{pdftitle={Hermitian manifolds with constant holomorphic sectional curvature or real bisectional curvature},pdfauthor={Kai Tang},pdfsubject={Research manuscript},pdfkeywords={Chern curvature, real bisectional curvature, pluriclosed metric}}
\numberwithin{equation}{section}
\newtheorem{theorem}{Theorem}[section]
\newtheorem{proposition}[theorem]{Proposition}
\newtheorem{lemma}[theorem]{Lemma}
\newtheorem{corollary}[theorem]{Corollary}
\newtheorem{conjecture}[theorem]{Conjecture}
\theoremstyle{definition}
\theoremstyle{remark}\newtheorem{remark}[theorem]{Remark}
\newcommand{\ii}{\sqrt{-1}}
\newcommand{\bp}{\bar\partial}
\newcommand{\ph}{\varphi}
\newcommand{\Jac}{\mathcal J}
\newcommand{\tr}{\operatorname{tr}}
\newcommand{\Ric}{\operatorname{Ric}}
\newcommand{\Rea}{\operatorname{Re}}
\newcommand{\eps}{\varepsilon}
\newcommand{\dV}{\mathrm dV_h}
\newcommand{\inner}[2]{\langle #1,#2\rangle}
\newcommand{\CH}{\mathbb{CH}}
\allowdisplaybreaks[2]
\title[Constant curvature on Hermitian manifolds]{Hermitian manifolds with constant holomorphic sectional curvature or real bisectional curvature}
\author{Kai Tang}
\address{School of Mathematical Sciences, Zhejiang Normal University, Jinhua, China}
\email{kaitang001@zjnu.edu.cn}
\makeatletter
\@namedef{subjclassname@2020}{\textup{2020} Mathematics Subject Classification}
\makeatother

\thanks{\text{Foundation item:} Supported by Natural Science Foundation of Zhejiang Province (No. LMS26A010006).}

\begin{document}
\begin{abstract}
An old conjecture in Hermitian geometry states that a compact Hermitian manifold with constant holomorphic sectional curvature is K\"ahler when the constant is nonzero and Chern flat when the constant is zero. Yang--Zheng introduced the real bisectional curvature as a generalization of the holomorphic sectional curvature and conjectured that, a compact Hermitian manifold with constant real bisectional curvature has zero constant and is Chern flat. They proved that the constant cannot be positive. In this paper, we first prove their conjecture. For a negative constant, we show that the metric is pluriclosed and use comparison with a negative K\"ahler--Einstein metric to obtain a contradiction. For the zero constant, we combine a torsion identity of Lin--Ren with the Bochner formula of Zhou--Zheng to prove Chern flatness. We also prove that a pluriclosed metric with constant negative holomorphic sectional curvature on a compact K\"ahler manifold is K\"ahler.
\end{abstract}
\maketitle

\section{Introduction}

The holomorphic sectional curvature plays a fundamental role in complex differential geometry. It is well known that a complete K\"ahler manifold with constant holomorphic sectional curvature is a complex space form. Its universal cover is, up to a scaling, the complex projective space, the complex Euclidean space, or the complex hyperbolic space (see \cite{ZhengBook}).

For a Hermitian manifold $(M^n,h)$, let $\nabla$, $T$, and $R$ denote the Chern connection, its torsion, and its curvature, respectively. The holomorphic sectional curvature is defined by
\[
 H_h(v)=\frac{R(v,\bar v,v,\bar v)}{|v|_h^4},
 \qquad 0\ne v\in T^{1,0}M.
\]
In the non-K\"ahler case, the Chern curvature tensor does not have all the K\"ahler symmetries, and the holomorphic sectional curvature does not determine the whole curvature tensor. This leads to the following long-standing conjecture.

\begin{conjecture}\label{conj:HSC}
Let $(M^n,h)$ be a compact Hermitian manifold with $n\ge2$. Suppose that $H_h\equiv c$ for a constant $c$. If $c\ne0$, then $h$ is K\"ahler. If $c=0$, then $h$ is Chern flat.
\end{conjecture}

Here Chern flatness means that Chern curvature tensor vanishes. Compact Chern-flat manifolds were studied by Boothby \cite{Boothby} in 1958, who proved that their universal covers are complex Lie groups equipped with left-invariant Hermitian metrics. Such a metric may not be K\"ahler. the standard Hermitian metric on the Iwasawa manifold is an example.

For $n=2$, Conjecture~\ref{conj:HSC} is known. The nonpositive case was established by Balas--Gauduchon \cite{BG}; see also Balas \cite{Balas}. The general case follows from the classification of compact self-dual Hermitian surfaces by Apostolov--Davidov--Mu\v skarov \cite{ADM}. For $n\geq3$, the conjecture is still open. Tang \cite{TangKL} verified the conjecture for Chern K\"ahler-like metrics (namely $R$ obeys all K\"{a}hler symmetries). Chen--Chen--Nie \cite{CCN} proved the nonpositive case for locally conformally K\"ahler metrics, and Huang--Wan \cite{HW} removed the sign restriction in this setting. Rao--Zheng \cite{RZ} established the conjecture for Bismut K\"ahler-like metrics. Further results on the conjecture can be found in the work of Li--Zheng \cite{LZ} on complex nilmanifolds and in recent work on Bismut torsion-parallel metrics by Chen--Zheng \cite{CZ}, Wang--Zheng \cite{WZ}, and Wang \cite{Wang}. In particular, \cite{WZ} treats balanced Bismut torsion-parallel fourfolds, whereas \cite{Wang} treats the nonzero constant case in arbitrary dimension. See Zheng \cite{ZhengSurvey} for a recent survey of the conjecture.

Broder--Tang \cite{BT} proved that a pluriclosed metric with vanishing holomorphic sectional curvature on a compact K\"ahler manifold is K\"ahler flat. They also proved K\"ahler flatness for Hermitian metrics with vanishing real bisectional curvature on compact complex manifolds in the Fujiki class~$\mathcal C$. In the first statement, the K\"ahler background and the pluriclosed metric need not coincide.

Motivated by extending the Wu--Yau theorem \cite{TY,WY} to Hermitian metrics, Yang--Zheng \cite{YZ} introduced the real bisectional curvature. For a unitary frame $e=(e_1,\ldots,e_n)$ and $a=(a_1,\ldots,a_n)\in\mathbb R^n_{\ge0}\setminus\{0\}$, it is defined by
\begin{equation}\label{eq:Bdef}
 B_h(e,a)=\frac{\sum_{i,j}R_{i\bar i j\bar j}a_i a_j}{\sum_i a_i^2}.
\end{equation}
For K\"ahler metrics, the sign conditions on $B_h$ and $H_h$ are equivalent. For general Hermitian metrics, the corresponding condition on $B_h$ is stronger. Constancy in \eqref{eq:Bdef} requires the same value at every point for every unitary frame.

For constant real bisectional curvature, Yang--Zheng proved the following result.
\begin{proposition}[Yang--Zheng {\cite[Theorem 1.4]{YZ}}]\label{prop:YZintro}
Let $(M^n,h)$ be a compact Hermitian manifold with $n\ge2$ and $B_h\equiv c$. Then $c\le0$. If $c=0$, then $h$ is balanced, its first, second, and third Chern--Ricci tensors vanish, and
\begin{equation}\label{eq:pairzero}
 R_{i\bar j k\bar\ell}+R_{k\bar\ell i\bar j}=0
\end{equation}
in every unitary frame.
\end{proposition}

Yang--Zheng proposed the following conjecture \cite{YZ} (also see \cite[Conjecture 2]{ZZ}).
\begin{conjecture}[Yang--Zheng]\label{conj:RBC}
Let $(M^n,h)$ be a compact Hermitian manifold with constant real bisectional curvature $B_h\equiv c$. Then $c=0$ and $h$ is Chern flat.
\end{conjecture}

Zhou--Zheng \cite{ZZ} proved Chern flatness in the zero case when $n=3$. Their proof uses the Bianchi identities, a Bochner formula for the torsion, and a special unitary frame available on balanced threefolds. Our first main result proves Conjecture~\ref{conj:RBC}.

\begin{theorem}\label{thm:RBC}
Let $(M^n,h)$ be a compact Hermitian manifold of complex dimension $n\ge2$. If $B_h\equiv c$ for a constant $c$, then
\[
 c=0,\qquad R\equiv0,\qquad \nabla T\equiv0.
\]
\end{theorem}

We first consider the case $c<0$. A contraction of the polarized curvature identity shows that $h$ is pluriclosed. It follows from Lee--Streets \cite{LS} that $M$ is K\"ahler and $K_M$ is ample. We may therefore choose a negative K\"ahler--Einstein metric $g$. Applying the Bochner formula to $\mathrm{id}:(M,g)\to(M,h)$ and then the maximum principle to $\log(\omega_g^n/\omega_h^n)$, we obtain $h=g$. This contradicts the K\"ahler curvature symmetries when $n\ge2$. The comparison is motivated by the proof of Broder--Tang \cite[Theorem 1.3]{BT}, where the Bochner formula is applied with a Ricci-flat K\"ahler background.
When $c=0$, the metric is balanced by Yang--Zheng, and the curvature identity gives
\[
 \ii\partial\bp\omega_h=\sum_a\tau^a\wedge\bar\tau^a,
\]
where $\tau^a$ are the torsion two-forms. Combining the pointwise identity of Lin--Ren \cite{LR} with Stokes' theorem, we deduce that the torsion satisfies the Jacobi identity. We then express the remaining term in the Bochner formula of Zhou--Zheng as a contraction of $\nabla''\Jac(T)$ with $T$. This term vanishes, and the maximum principle gives $\nabla T=0$. The first Bianchi identity then implies that $R=0$.

The comparison argument also applies to constant holomorphic sectional curvature when the metric is pluriclosed. Our second main result is the following.
\begin{theorem}\label{thm:HSC}
Let $X$ be a compact complex manifold of complex dimension $n\ge2$ which admits a K\"ahler metric. Suppose that $h$ is a pluriclosed Hermitian metric with $H_h\equiv c<0$. Then $h$ is K\"ahler. Consequently, $(X,h)$ is a compact complex hyperbolic space form.
\end{theorem}

More precisely, if $a=(n+1)c/2$ and $g$ is the K\"ahler--Einstein metric normalized by $\Ric_g^{(1)}=a\omega_g$, then $h=g$. The existence of $g$ follows from Broder--Stanfield \cite{BS} and Aubin--Yau \cite{Aubin,Yau}. Set $u=\tr_g h$ and $v=\tr_h g$. In the Bochner formula for $u$, the pluriclosed curvature identity produces a torsion term which cancels exactly the skew-symmetric part of $\nabla\partial\mathrm{id}$. The remaining terms give
\[
 \Delta_g u\ge\frac{(-c)(n+1)}{2n}u(u-n).
\]
The maximum principle yields $u\le n$, and the inequality $uv\ge n^2$ gives $v\ge n$. On the other hand, for the logarithm of the volume ratio,
\[
 \Delta_h\log\frac{\omega_g^n}{\omega_h^n}
 =(-a)(v-n)+\frac14|T|_h^2\ge0.
\]
A second application of the maximum principle gives $T=0$ and $v=n$. Equality in $uv\ge n^2$ then implies that $h=g$.

Combining Theorem~\ref{thm:HSC} with \cite[Corollary 1.2]{BT}, we obtain:
\begin{corollary}\label{cor:HSC}
Let $X$ be a compact complex manifold admitting a K\"ahler metric, and let $h$ be a pluriclosed Hermitian metric with $H_h\equiv c\le0$. Then $h$ is K\"ahler. If $c<0$, it is a complex hyperbolic space form metric. If $c=0$, it is K\"ahler flat, and $X$ admits a finite \emph{\'etale} cover by a complex torus.
\end{corollary}

The paper is organized as follows. Section~\ref{sec:prelim} fixes the conventions and proves the torsion identities needed below, including the pointwise formula of Lin--Ren. Section~\ref{sec:RBC} proves Theorem~\ref{thm:RBC} and gives K\"ahler-flatness consequences when additional special metrics exist. Section~\ref{sec:HSC} proves Theorem~\ref{thm:HSC}.

\section{Preliminaries}\label{sec:prelim}

\subsection{The Chern connection and special Hermitian metrics}
Throughout the paper, manifolds are assumed connected and without boundary. All curvature and torsion tensors are those of the Chern connection unless otherwise indicated. In holomorphic coordinates, write
\[
 \omega=\omega_h=\ii h_{i\bar j}\,dz^i\wedge d\bar z^j,
 \qquad \dV=\frac{\omega^n}{n!}.
\]
The Chern connection, Chern torsion and curvature are normalized by
\begin{align}
 \Gamma^k_{ij}&=h^{k\bar\ell}\partial_i h_{j\bar\ell},
 &T^k_{ij}&=\Gamma^k_{ij}-\Gamma^k_{ji},\label{eq:connection}\\
 R_{i\bar j k\bar\ell}
 &=-\partial_i\partial_{\bar j}h_{k\bar\ell}
 +h^{p\bar q}(\partial_i h_{k\bar q})(\partial_{\bar j}h_{p\bar\ell}).\label{eq:curv}
\end{align}
The curvature satisfies
\begin{equation}\label{eq:herm}
 \overline{R_{i\bar j k\bar\ell}}=R_{j\bar i\ell\bar k}.
\end{equation}
In a local unitary frame with dual coframe $\ph^1,\ldots,\ph^n$, set
\begin{equation}\label{eq:torsionforms}
 \tau^a=\sum_{i<k}T^a_{ik}\ph^i\wedge\ph^k,
 \qquad \eta=\sum_i\eta_i\ph^i,
 \qquad \eta_i=\sum_r T^r_{ri}.
\end{equation}
The structure equations give
\begin{equation}\label{eq:dw}
 \partial\omega=\ii\sum_a\tau^a\wedge\bar\ph^a,
 \qquad \partial\omega^{n-1}=-\eta\wedge\omega^{n-1}.
\end{equation}
Indeed, at a point we may choose a smooth unitary frame whose Chern connection form vanishes. Then $d\ph^a=\tau^a$ at that point, and differentiating $\omega=\ii\sum_a\ph^a\wedge\bar\ph^a$ gives the first identity in \eqref{eq:dw}. We will use such frames for tensor calculations; holomorphic normal coordinates will be chosen only for a K\"ahler metric.

In particular, $h$ is balanced, meaning $d\omega^{n-1}=0$, if and only if $\eta=0$. It is pluriclosed if $\partial\bp\omega=0$, and Gauduchon if $\partial\bp\omega^{n-1}=0$. A Hermitian form $\alpha$ is called astheno-K\"ahler if $\partial\bp\alpha^{n-2}=0$.

We also have tensor norms:
\[
 |T|^2=\sum_{i,k,a}|T^a_{ik}|^2,
 \quad |\nabla'T|^2=\sum_{i,k,a,s}|T^a_{ik,s}|^2,
 \quad |\nabla''T|^2=\sum_{i,k,a,s}|T^a_{ik,\bar s}|^2.
\]
On exterior forms, increasing multi-indices give an orthonormal basis. Thus $\sum_a|\tau^a|^2=|T|^2/2$. Commas denote successive Chern covariant derivatives in the order in which the indices are written. We set $\Delta_h f=h^{i\bar j}\partial_i\partial_{\bar j}f$.

The four Chern--Ricci tensors are defined, in a unitary frame, by
\begin{align*}
 \Ric^{(1)}_{i\bar j}&=\sum_rR_{i\bar j r\bar r},
 &\Ric^{(2)}_{i\bar j}&=\sum_rR_{r\bar r i\bar j},\\
 \Ric^{(3)}_{i\bar j}&=\sum_rR_{i\bar r r\bar j},
 &\Ric^{(4)}_{i\bar j}&=\sum_rR_{r\bar j i\bar r}.
\end{align*}
By \eqref{eq:herm}, the first two tensors are Hermitian and
\[
 \Ric^{(4)}_{i\bar j}=\overline{\Ric^{(3)}_{j\bar i}}.
\]
We use the same notation for the associated $(1,1)$-forms,
\[
 \Ric_h^{(q)}=\ii\sum_{i,j}\Ric^{(q)}_{i\bar j}\ph^i\wedge\bar\ph^j,
 \qquad 1\le q\le4.
\]
Thus the first two forms are real and $\overline{\Ric_h^{(3)}}=\Ric_h^{(4)}$; the last two need not be real individually. In particular,
\[
 \Ric_h^{(1)}=-\ii\partial\bp\log\det h,
 \qquad s_h=\tr_h\Ric_h^{(1)}=\sum_{i,k}R_{i\bar i k\bar k}.
\]
For a $(1,1)$-form $\alpha=\ii a_{i\bar j}\ph^i\wedge\bar\ph^j$, we write $\tr_h\alpha=\sum_i a_{i\bar i}$. For a K\"ahler metric the four Ricci tensors coincide.

The first Bianchi identity, in the convention \eqref{eq:connection}, is
\begin{equation}\label{eq:Bianchi1}
 T^\ell_{ik,\bar j}=R_{k\bar j i\bar\ell}-R_{i\bar j k\bar\ell}.
\end{equation}
Tracing $k=\ell$ and using $\sum_kT^k_{ik}=-\eta_i$ gives
\[
 \Ric^{(4)}_{i\bar j}-\Ric^{(1)}_{i\bar j}=-\eta_{i,\bar j}.
\]
Consequently, as identities of forms,
\begin{equation}\label{eq:Riceta}
 \Ric_h^{(4)}-\Ric_h^{(1)}=\ii\bp\eta,
 \qquad
 \Ric_h^{(3)}-\Ric_h^{(1)}=-\ii\partial\bar\eta.
\end{equation}
Here we have used $\bp\eta=-\sum_{i,j}\eta_{i,\bar j}\ph^i\wedge\bar\ph^j$; the second identity follows by conjugating the first.

\subsection{Constant real bisectional curvature}
\begin{lemma}\label{lem:polar}
If $B_h\equiv c$, then in a unitary frame
\begin{equation}\label{eq:polar}
 R_{i\bar j k\bar\ell}+R_{k\bar\ell i\bar j}
 =2c\delta_{i\ell}\delta_{kj}.
\end{equation}
In particular, $H_h\equiv c$ and $s_h=nc$.
\end{lemma}
\begin{proof}
This is the polarization identity of Yang--Zheng \cite[\S3]{YZ}; we recall the argument for completeness. Fix a point and a unitary frame $e=(e_1,\ldots,e_n)$. For a Hermitian matrix $\xi=(\xi_{k\ell})$, set
\[
 \mathcal Q(\xi)=\sum_{k,\ell,s,t}R_{k\bar\ell s\bar t}\xi_{k\ell}\xi_{st}.
\]
Let $A=(A_{ik})$ be a unitary matrix and set $f_i=\sum_k A_{ik}e_k$. For real numbers $v_1,\ldots,v_n$, write
\[
 \xi_{k\ell}=\sum_i v_iA_{ik}\overline{A_{i\ell}}.
\]
The matrix $\xi$ is Hermitian. Expanding in the fixed frame gives
\begin{align*}
 &\sum_{i,j}R(f_i,\bar f_i,f_j,\bar f_j)v_iv_j\\
 &\quad=\sum_{i,j,k,\ell,s,t}v_iv_jA_{ik}\overline{A_{i\ell}}
 A_{js}\overline{A_{jt}}R_{k\bar\ell s\bar t}\\
 &\quad=\sum_{k,\ell,s,t}\xi_{k\ell}\xi_{st}R_{k\bar\ell s\bar t}
 =\mathcal Q(\xi).
\end{align*}
Moreover, unitarity yields
\begin{align*}
 \tr(\xi^2)
 &=\sum_{k,\ell}\xi_{k\ell}\xi_{\ell k}\\
 &=\sum_{i,j}v_iv_j
 \left(\sum_kA_{ik}\overline{A_{jk}}\right)
 \left(\sum_\ell A_{j\ell}\overline{A_{i\ell}}\right)
 =\sum_i v_i^2.
\end{align*}
For $v_i>0$, the definition of $B_h$ therefore gives
\begin{equation}\label{eq:matrixRBC}
 \mathcal Q(\xi)=c\tr(\xi^2).
\end{equation}
Both sides are quadratic polynomials on the real vector space of Hermitian matrices. Hence \eqref{eq:matrixRBC} holds for every Hermitian matrix $\xi$.

For Hermitian matrices $\xi,\zeta$, polarization now gives
\begin{align*}
 &\sum_{i,j,k,\ell}(R_{i\bar j k\bar\ell}+R_{k\bar\ell i\bar j})\xi_{ij}\zeta_{k\ell}\\
 &\qquad=\mathcal Q(\xi+\zeta)-\mathcal Q(\xi)-\mathcal Q(\zeta)
 =2c\tr(\xi\zeta).
\end{align*}
Every complex matrix is a complex linear combination of Hermitian matrices. Extend the last bilinear identity complex linearly and take $\xi=E_{ij}$ and $\zeta=E_{k\ell}$. Since $\tr(E_{ij}E_{k\ell})=\delta_{jk}\delta_{i\ell}$, we obtain \eqref{eq:polar}.

Taking $a=(1,0,\ldots,0)$ and $a=(1,\ldots,1)$ in \eqref{eq:Bdef}, we obtain $H_h=c$ and $s_h=nc$, respectively.
\end{proof}

\begin{lemma}\label{lem:ddbar}
Set $\Sigma=\sum_a\tau^a\wedge\bar\tau^a$. If $B_h\equiv c$, then
\begin{equation}\label{eq:ddbar}
 \ii\partial\bp\omega=c\omega^2+\Sigma.
\end{equation}
\end{lemma}
\begin{proof}
In holomorphic coordinates define
\begin{align*}
 D_{i\bar j k\bar\ell}
 &=\partial_i\partial_{\bar j}h_{k\bar\ell}
 +\partial_k\partial_{\bar\ell}h_{i\bar j}
 -\partial_i\partial_{\bar\ell}h_{k\bar j}
 -\partial_k\partial_{\bar j}h_{i\bar\ell},\\
 C_{i\bar j k\bar\ell}
 &=R_{i\bar j k\bar\ell}+R_{k\bar\ell i\bar j}
 -R_{i\bar\ell k\bar j}-R_{k\bar j i\bar\ell}.
\end{align*}
It follows from \eqref{eq:curv} that
\begin{equation}\label{eq:fourterm}
 D_{i\bar j k\bar\ell}=h_{p\bar q}T^p_{ik}\overline{T^q_{j\ell}}-C_{i\bar j k\bar\ell}.
\end{equation}
Here the terms involving first derivatives are
\[
 h^{p\bar q}(\partial_i h_{k\bar q}-\partial_k h_{i\bar q})
 (\partial_{\bar j}h_{p\bar\ell}-\partial_{\bar\ell}h_{p\bar j}).
\]
Also,
\[
 \ii\partial\bp\omega
 =\sum_{i<k,\,j<\ell}D_{i\bar j k\bar\ell}\,
 dz^i\wedge dz^k\wedge d\bar z^j\wedge d\bar z^\ell.
\]
At a fixed point, choose the coordinates so that $h_{i\bar j}=\delta_{ij}$. By \eqref{eq:polar},
\[
 C_{i\bar j k\bar\ell}=2c(\delta_{i\ell}\delta_{kj}-\delta_{ij}\delta_{k\ell}).
\]
Since $\omega^2=2\sum_{i<k}\ph^i\wedge\ph^k\wedge\bar\ph^i\wedge\bar\ph^k$, formula \eqref{eq:fourterm} proves \eqref{eq:ddbar}.
\end{proof}

\subsection{The torsion Jacobi identity}
Define the Jacobi tensor by
\begin{equation}\label{eq:Jacdef}
 \Jac^a_{ijk}=\sum_b\left(T^b_{ij}T^a_{bk}+T^b_{jk}T^a_{bi}+T^b_{ki}T^a_{bj}\right).
\end{equation}
It is alternating in $i,j,k$. Let
\[
 J^a=\sum_{i<j<k}\Jac^a_{ijk}\ph^i\wedge\ph^j\wedge\ph^k.
\]
In particular, $\sum_a|J^a|^2=\frac16\sum_{a,i,j,k}|\Jac^a_{ijk}|^2$.

We shall use the pointwise identity in \cite[Lemma 2.2]{LR}. For completeness, we also give the proof.
\begin{lemma}[Lin--Ren]\label{lem:LR}
Let $T\in\Lambda^2 V^*\otimes V$ on an $n$-dimensional Hermitian vector space satisfy $\sum_bT^b_{bj}=0$. Define $\tau^a$, $\Sigma$, and $J^a$ as above, and put
\[
 U=\ii\sum_b\tau^b\wedge\bar\ph^b,
 \qquad V_0=\bar U=-\ii\sum_b\ph^b\wedge\bar\tau^b.
\]
For $n\ge4$, define
\[
 Q_n=\begin{cases}
 \Sigma^2,&n=4,\\[3pt]
 \displaystyle\Sigma^2\wedge\frac{\omega^{n-4}}{(n-4)!}
 +\ii\Sigma\wedge U\wedge V_0\wedge\frac{\omega^{n-5}}{(n-5)!},&n\ge5.
 \end{cases}
\]
Then
\begin{equation}\label{eq:LR}
 Q_n=\sum_a|J^a|^2\frac{\omega^n}{n!}.
\end{equation}
For $n\le3$, the same trace condition implies $\Jac=0$.
\end{lemma}
\begin{proof}
Let $\iota_b$ be contraction with $e_b$, let $\eps_b\alpha=\ph^b\wedge\alpha$, and set $W^{ab}=\tau^a\wedge\tau^b$. Direct expansion of \eqref{eq:Jacdef} gives
\[
 J^a=\sum_b\tau^b\wedge\iota_b\tau^a.
\]
Since $\sum_b\iota_b\tau^b=\eta=0$, the contraction rule gives
\begin{equation}\label{eq:JW}
 J^a=\sum_b\iota_bW^{ab}.
\end{equation}
Take the Hermitian inner product to be linear in its first variable. Using $\eps_b^*=\iota_b$ and $\iota_c\eps_b+\eps_b\iota_c=\delta_{bc}\mathrm{Id}$, we find
\begin{align}
 \sum_a|J^a|^2
 &=\sum_{a,b,c}\inner{\iota_bW^{ab}}{\iota_cW^{ac}}\notag\\
 &=\sum_{a,b,c}\inner{W^{ab}}{\eps_b\iota_cW^{ac}}\notag\\
 &=\sum_{a,b}|W^{ab}|^2
 -\sum_{a,b,c}\inner{\eps_cW^{ab}}{\eps_bW^{ac}}.\label{eq:creation}
\end{align}
The last equality follows from $\eps_b\iota_c=\delta_{bc}-\iota_c\eps_b$.

For $(r,0)$-forms $\alpha,\beta$, the unitary Hodge identity is
\begin{equation}\label{eq:Hodge}
 \ii^{\,r^2}\alpha\wedge\bar\beta\wedge\frac{\omega^{n-r}}{(n-r)!}
 =\inner{\alpha}{\beta}\frac{\omega^n}{n!}.
\end{equation}
Applying \eqref{eq:Hodge} with $r=4$ and using $\Sigma^2=\sum_{a,b}W^{ab}\wedge\overline{W^{ab}}$, we obtain
\begin{equation}\label{eq:Wnorm}
 \Sigma^2\wedge\frac{\omega^{n-4}}{(n-4)!}
 =\sum_{a,b}|W^{ab}|^2\frac{\omega^n}{n!}.
\end{equation}
Rearranging the factors in the mixed term gives
\begin{equation}\label{eq:mixed}
 \ii\Sigma\wedge U\wedge V_0
 =-\ii\sum_{a,b,c}\eps_cW^{ab}\wedge\overline{\eps_bW^{ac}}.
\end{equation}
The minus sign is due to moving $\ph^c$ past $\bar\tau^a\wedge\bar\ph^b$, which has degree three. For $n\ge5$, apply \eqref{eq:Hodge} with $r=5$ to \eqref{eq:mixed}; together with \eqref{eq:Wnorm} and \eqref{eq:creation}, this gives \eqref{eq:LR}. For $n=4$, the five-forms $\eps_cW^{ab}$ vanish, and the same conclusion follows from \eqref{eq:creation} and \eqref{eq:Wnorm}.

If $n=3$, write $a=T(e_2,e_3)$, $b=T(e_3,e_1)$, and $c=T(e_1,e_2)$, with components $a_r,b_r,c_r$. The trace condition says $b_3=c_2$, $a_3=c_1$, and $b_1=a_2$. Hence
\[
 \Jac(e_1,e_2,e_3)
 =(c_2-b_3)a+(a_3-c_1)b+(b_1-a_2)c=0.
\]
For $n\le2$, the assertion follows from $\Lambda^3V^*=0$.
\end{proof}

\begin{remark}\label{rem:LRscope}
Lemma~\ref{lem:LR} requires only skew-symmetry and the trace condition $\eta=0$, the latter being used in \eqref{eq:JW}. Thus it applies without the Chern K\"ahler-like assumption or the condition $\nabla''T=0$. In our norm convention,

\[
 \sum_a|J^a|^2\,\dV=\frac16\sum_{a,i,j,k}|\Jac^a_{ijk}|^2\,\dV.
\]
The mixed term in $Q_n$ cannot be omitted when $n\ge5$. For example, on $\mathbb C^5$ take $T(e_1,e_2)=e_3$, $T(e_3,e_4)=e_5$, extend by skew-symmetry, and set all other components to zero. This tensor has zero trace, and direct expansion gives
\[
 \Sigma^2\wedge\omega=2\,\dV,\qquad
 \ii\Sigma\wedge U\wedge V_0=-\dV,\qquad
 \sum_a|J^a|^2=1.
\]
This algebraic example also shows that the zero trace condition alone does not imply the Jacobi identity in higher dimensions.
\end{remark}

\begin{proposition}\label{prop:Jac}
Let $(M^n,h)$ be compact and balanced. If
\begin{equation}\label{eq:Sigmaexact}
 \ii\partial\bp\omega=\Sigma,
\end{equation}
then its Chern torsion satisfies $\Jac(T)=0$.
\end{proposition}
\begin{proof}
Since $h$ is balanced, $\eta=0$. Lemma~\ref{lem:LR} proves the assertion when $n\le3$, so we assume that $n\ge4$. By \eqref{eq:Sigmaexact}, we have $\partial\Sigma=\bp\Sigma=0$. For any positive integer $m$,

\begin{equation}\label{eq:power}
 \partial\bp\omega^m
 =m\partial\bp\omega\wedge\omega^{m-1}
 +m(m-1)\partial\omega\wedge\bp\omega\wedge\omega^{m-2}.
\end{equation}
Indeed, by the product rule,
\begin{align*}
 \partial\bp\omega^m
 &=m\partial(\bp\omega\wedge\omega^{m-1})\\
 &=m\partial\bp\omega\wedge\omega^{m-1}
 -m(m-1)\bp\omega\wedge\partial\omega\wedge\omega^{m-2}.
\end{align*}
Since $\bp\omega\wedge\partial\omega=-\partial\omega\wedge\bp\omega$, this gives \eqref{eq:power}.

For $n=4$, we have $Q_4=\ii\partial\bp(\Sigma\wedge\omega)$. For $n\ge5$, using \eqref{eq:dw}, \eqref{eq:Sigmaexact}, and \eqref{eq:power}, we obtain
\begin{align*}
 \ii\partial\bp\left(\Sigma\wedge\frac{\omega^{n-3}}{(n-3)!}\right)
 &=\Sigma^2\wedge\frac{\omega^{n-4}}{(n-4)!}
 +\ii\Sigma\wedge\partial\omega\wedge\bp\omega
 \wedge\frac{\omega^{n-5}}{(n-5)!}\\
 &=Q_n.
\end{align*}
The form $\Sigma$ is invariant under unitary changes of frame. Hence
\[
 \gamma=\Sigma\wedge\frac{\omega^{n-3}}{(n-3)!}
\]
is a globally defined $(n-1,n-1)$-form and $Q_n=\ii\partial\bp\gamma=d(\ii\bp\gamma)$. By Stokes' theorem and Lemma~\ref{lem:LR},
\[
 0=\int_M Q_n=\int_M\sum_a|J^a|^2\,\dV.
\]
It follows that $J^a=0$ for all $a$. Thus $\Jac(T)=0$ on $M$.

\end{proof}

\section{Constant real bisectional curvature}\label{sec:RBC}

\subsection{Excluding a negative constant}
\begin{lemma}\label{lem:auto}
On a Hermitian manifold, if $B_h\equiv c$ for a nonzero constant $c$, then $h$ is pluriclosed.
\end{lemma}
\begin{proof}
Contract \eqref{eq:polar} with $(i,j,k,\ell)$ replaced by $(r,j,i,r)$. Summing over $r$ gives
\[
 \Ric^{(4)}_{i\bar j}+\Ric^{(3)}_{i\bar j}=2nc\delta_{ij}.
\]
Add the two identities in \eqref{eq:Riceta}. As identities of forms, this yields
\begin{align*}
 \ii(\bp\eta-\partial\bar\eta)
 &=\Ric_h^{(3)}+\Ric_h^{(4)}-2\Ric_h^{(1)}\\
 &=2nc\omega-2\Ric_h^{(1)}.
\end{align*}
Equivalently,
\begin{equation}\label{eq:ricciform}
 nc\omega=\Ric_h^{(1)}+\frac{\ii}{2}(\bp\eta-\partial\bar\eta).
\end{equation}
Apply $\partial\bp$. The first Ricci form is closed, while $\partial\bp(\bp\eta)=\partial\bp(\partial\bar\eta)=0$. Since $c$ is constant and nonzero, we obtain $\partial\bp\omega=0$.
\end{proof}

We first give the comparison identity used in both main theorems.
\begin{lemma}\label{lem:CL}
Let $g$ be a K\"ahler metric with $\Ric_g^{(1)}=a\omega_g$, where $a$ is constant, and let $h$ be a Hermitian metric. Put $u=\tr_g h$. Then
\begin{equation}\label{eq:CL}
 \Delta_g u=|\nabla\partial\mathrm{id}|^2+au-Q,\qquad
 Q=g^{i\bar j}g^{k\bar\ell}R^h_{i\bar j k\bar\ell}.
\end{equation}
Here $\mathrm{id}:(M,g)\to(M,h)$ is the identity map, and $\nabla$ is the connection induced on $(T^{1,0}M)^*\otimes\mathrm{id}^*T^{1,0}M$. The norm contracts the two input indices by $g$ and the output index by $h$.
\end{lemma}
\begin{proof}
Apply the Chern--Lu formula \cite[Lemma 4.1]{YZ} to $\mathrm{id}$. Since $g$ is K\"ahler and $\Ric_g^{(1)}=a\omega_g$, the source Ricci term is $au$, while the target curvature term is $Q$. For later calculations, choose holomorphic coordinates at a point such that
\begin{equation}\label{eq:diag}
 g_{i\bar j}=\delta_{ij},\qquad \partial g=0,\qquad h_{i\bar j}=\lambda_i\delta_{ij}.
\end{equation}
In these coordinates,
\[
 (\nabla\partial\mathrm{id})^p_{ik}=\Gamma(h)^p_{ik}-\Gamma(g)^p_{ik}
 =\lambda_p^{-1}\partial_i h_{k\bar p},
 \qquad
 |\nabla\partial\mathrm{id}|^2=\sum_{i,k,p}\lambda_p^{-1}|\partial_i h_{k\bar p}|^2.
\]
This gives the stated norm convention in \eqref{eq:CL}.
\end{proof}

\begin{proposition}\label{prop:c0}
Let $(M^n,h)$ be compact, with $n\ge2$. If $B_h\equiv c$ is constant, then $c=0$.
\end{proposition}
\begin{proof}
By Proposition~\ref{prop:YZintro}, $c\le0$. Suppose that $c<0$. Lemma~\ref{lem:auto} shows that $h$ is pluriclosed. By the argument of Lee--Streets \cite{LS}, the existence of the pluriclosed metric $h$ with strictly negative real bisectional curvature implies that $M$ is K\"ahler and $K_M$ is ample.

By the Aubin--Yau theorem \cite{Aubin,Yau}, there is a K\"ahler metric $g$, scaled so that
\[
 \Ric_g^{(1)}=c\omega_g.
\]
Set $u=\tr_g h$, $v=\tr_h g$, and $q=\sum_i\lambda_i^2$, where the $\lambda_i$ are the eigenvalues of $h$ with respect to $g$. We first show that $u\le n$. In the coordinates \eqref{eq:diag}, the frame $E_i=\lambda_i^{-1/2}\partial_i$ is $h$-unitary. Hence
\[
 Q=\sum_{i,k}\lambda_i\lambda_k R^h(E_i,\bar E_i,E_k,\bar E_k)
 =c\sum_i\lambda_i^2=cq.
\]
Lemma~\ref{lem:CL} therefore gives
\begin{equation}\label{eq:RBCtrace}
 \Delta_g u=|\nabla\partial\mathrm{id}|^2+cu-cq
 \ge(-c)\left(\frac{u^2}{n}-u\right).
\end{equation}
At a maximum point of $u$, the left-hand side is nonpositive. Since $u>0$, it follows that $u\le n$ everywhere. By Cauchy--Schwarz,
\begin{equation}\label{eq:uv}
 uv=\left(\sum_i\lambda_i\right)\left(\sum_i\lambda_i^{-1}\right)\ge n^2,
 \qquad v\ge n.
\end{equation}

Next, put $F=\log(\omega_g^n/\omega_h^n)$. Since $s_h=nc$ by Lemma~\ref{lem:polar},
\begin{equation}\label{eq:RBCvolume}
 \Delta_h F=s_h-\tr_h\Ric_g^{(1)}=nc-cv=(-c)(v-n)\ge0.
\end{equation}
Here $F$ is globally defined and satisfies
\[
 \ii\partial\bp F=\Ric_h^{(1)}-\Ric_g^{(1)}.
\]
The strong maximum principle makes $F$ constant. Equations~\eqref{eq:RBCvolume} and \eqref{eq:uv} therefore give
\[
 v=n,\qquad n^2\le uv=nu\le n^2,
 \qquad u=n,\qquad
 \lambda_1=\cdots=\lambda_n=1.
\]
The last implication is the equality case of Cauchy--Schwarz applied to $(\sqrt{\lambda_i})_i$ and $(\lambda_i^{-1/2})_i$. Hence $h=g$.

Since $h=g$ is K\"ahler, its pair symmetry and \eqref{eq:polar} imply $R^h_{i\bar j k\bar\ell}=c\delta_{i\ell}\delta_{kj}$. For $i\ne k$, this says
\[
 R^h_{i\bar i k\bar k}=0,\qquad R^h_{k\bar i i\bar k}=c,
\]
which contradicts the K\"ahler symmetry because $c<0$. Hence $c=0$.
\end{proof}

\subsection{The zero case and the torsion Bochner formula}
We assume that $M$ is compact and $B_h=0$. By Proposition~\ref{prop:YZintro}, the metric is balanced and
\begin{equation}\label{eq:Riczero}
 \Ric^{(1)}=\Ric^{(2)}=\Ric^{(3)}=\Ric^{(4)}=0.
\end{equation}
Lemma~\ref{lem:ddbar} and Proposition~\ref{prop:Jac} immediately give
\begin{equation}\label{eq:Jaczero}
 \ii\partial\bp\omega=\Sigma,\qquad \Jac(T)=0.
\end{equation}
We now apply the torsion Bochner formula of Zhou--Zheng \cite{ZZ}.

\begin{lemma}[Zhou--Zheng]\label{lem:derivs}
Let $(M^n,h)$ be a compact Hermitian manifold with $B_h\equiv0$. Then the following identities hold in every unitary frame:
\begin{align}
 T^\ell_{ik,\bar j}&=T^j_{ik,\bar\ell},\label{eq:symT}\\
 \sum_s T^j_{is,\bar s}&=0,\label{eq:divT}\\
 2R_{i\bar j k\bar\ell}
 &=-T^\ell_{ik,\bar j}-\overline{T^k_{j\ell,\bar i}},\label{eq:recover}\\
 \sum_sT^j_{ik,\bar s s}
 &=\sum_{r,s}T^r_{is}T^j_{rk,\bar s},\label{eq:traceT}\\
 \sum_sT^j_{ik,s\bar s}&=\sum_sT^j_{ik,\bar s s}.\label{eq:commtrace}
\end{align}
\end{lemma}
\begin{proof}
These identities follow from the Bianchi calculations in \cite{ZZ}. For completeness, we recall the relevant contractions. The first Bianchi identity \eqref{eq:Bianchi1} and the pair antisymmetry \eqref{eq:pairzero} give \eqref{eq:symT} and \eqref{eq:recover}. Tracing \eqref{eq:symT} gives
\[
 \sum_sT^j_{is,\bar s}=\sum_sT^s_{is,\bar j}=-\eta_{i,\bar j}=0.
\]
The second Bianchi identity together with \eqref{eq:recover} gives
\[
 T^\ell_{mk,\bar j i}-T^\ell_{ik,\bar j m}
 =-\sum_rT^r_{im}T^\ell_{rk,\bar j}.
\]
Setting $m=j=s$ and summing, the first term vanishes by \eqref{eq:divT}; the remaining terms yield \eqref{eq:traceT}. Finally, the mixed commutation formula is
\begin{align*}
 \sum_s(T^j_{ik,\bar s s}-T^j_{ik,s\bar s})
 =\sum_r\bigl(T^r_{ik}\Ric^{(2)}_{r\bar j}
 -T^j_{rk}\Ric^{(2)}_{i\bar r}
 -T^j_{ir}\Ric^{(2)}_{k\bar r}\bigr)=0.
\end{align*}
The mixed torsion is zero, and all three Ricci terms vanish by \eqref{eq:Riczero}. This proves \eqref{eq:commtrace}.
\end{proof}

\begin{lemma}\label{lem:Bochner}
Let $(M^n,h)$ be a compact Hermitian manifold with $B_h\equiv0$. Then
\begin{equation}\label{eq:Bochner}
 \Delta_h|T|^2=|\nabla'T|^2+|\nabla''T|^2
 +\Rea\sum_{i,k,j,s}\Jac^j_{isk,\bar s}\overline{T^j_{ik}}.
\end{equation}
\end{lemma}
\begin{proof}
Put $f=|T|^2$. By metric compatibility,
\[
 f_{,s}=\sum_{i,k,j}\left(T^j_{ik,s}\overline{T^j_{ik}}
 +T^j_{ik}\overline{T^j_{ik,\bar s}}\right).
\]
Differentiating in the direction $\bar e_s$ and summing over $s$, we obtain
\begin{align*}
 \Delta_h f=\sum_{i,k,j,s}\bigl(&T^j_{ik,s\bar s}\overline{T^j_{ik}}
 +|T^j_{ik,s}|^2+|T^j_{ik,\bar s}|^2
 +T^j_{ik}\overline{T^j_{ik,\bar s s}}\bigr).
\end{align*}
By \eqref{eq:commtrace}, the two contracted second-derivative terms are complex conjugates. It follows that
\begin{align}
 \Delta_h|T|^2
 &=|\nabla'T|^2+|\nabla''T|^2
 +2\Rea\sum_{i,k,j,s}T^j_{ik,\bar s s}\overline{T^j_{ik}}\notag\\
 &=|\nabla'T|^2+|\nabla''T|^2+\mathcal P,\label{eq:BochnerZZ}
\end{align}
where, by \eqref{eq:traceT},
\[
 \mathcal P=2\Rea\sum_{i,k,j,r,s}T^r_{is}T^j_{rk,\bar s}\overline{T^j_{ik}}.
\]
This is the torsion Bochner formula obtained in \cite{ZZ}.

Put $D^j_{ik}=\sum_{r,s}T^r_{is}T^j_{rk,\bar s}$. By \eqref{eq:divT},
$D^j_{ik}=\sum_{r,s}(T^r_{is}T^j_{rk})_{,\bar s}$. We claim that
\begin{equation}\label{eq:Jacdiv}
 \sum_s\Jac^j_{isk,\bar s}=D^j_{ik}-D^j_{ki}.
\end{equation}
By \eqref{eq:divT}, the first product in \eqref{eq:Jacdef} satisfies
\[
 \sum_{r,s}(T^r_{is}T^j_{rk})_{,\bar s}
 =\sum_r\left(\sum_sT^r_{is,\bar s}\right)T^j_{rk}
 +\sum_{r,s}T^r_{is}T^j_{rk,\bar s}=D^j_{ik}.
\]
Similarly, the second product contributes $-D^j_{ki}$, since $T^r_{sk}=-T^r_{ks}$. For the third product, we have
\[
 \sum_{r,s}(T^r_{ki}T^j_{rs})_{,\bar s}
 =\sum_{r,s}T^r_{ki,\bar s}T^j_{rs}
 +\sum_rT^r_{ki}\left(\sum_sT^j_{rs,\bar s}\right)=0.
\]
Indeed, the last sum vanishes by \eqref{eq:divT}. The first sum vanishes by \eqref{eq:symT}, because $T^r_{ki,\bar s}$ is symmetric in $r,s$ and $T^j_{rs}$ is skew-symmetric. This proves \eqref{eq:Jacdiv}. Contracting with $\overline{T^j_{ik}}$ and interchanging $i,k$ in the second term, we obtain
\[
 \mathcal P=\Rea\sum_{i,k,j,s}\Jac^j_{isk,\bar s}\overline{T^j_{ik}},
\]
which proves \eqref{eq:Bochner}.
\end{proof}

\begin{proof}[Proof of Theorem~\ref{thm:RBC}]
By Proposition~\ref{prop:c0}, $c=0$. Equation~\eqref{eq:Jaczero} implies that $\Jac(T)=0$ and hence $\nabla\Jac(T)=0$. It follows from Lemma~\ref{lem:Bochner} that
\[
 \Delta_h|T|^2=|\nabla'T|^2+|\nabla''T|^2\ge0.
\]
The strong maximum principle gives $|T|^2=\mathrm{const}$. Hence $\nabla'T=\nabla''T=0$, and \eqref{eq:recover} yields $R=0$. This completes the proof.
\end{proof}

\begin{remark}
On a balanced threefold the Jacobi identity follows directly from the trace condition, as in Lemma~\ref{lem:LR}. In higher dimensions, Proposition~\ref{prop:Jac} uses in addition the identity $\ii\partial\bp\omega=\Sigma$ and compactness.
\end{remark}

\subsection{Additional special metrics}
In this section, we give the K\"ahler-flatness consequences of the torsion identity. 
\begin{corollary}\label{cor:special}
Let $(M^n,h)$ be compact, with $n\ge2$ and constant real bisectional curvature. Then $h$ is K\"ahler flat if any of the following holds:
\begin{enumerate}
 \item $M$ admits a pluriclosed Hermitian metric;
 \item $M$ admits an astheno-K\"ahler Hermitian metric;
\end{enumerate}
\end{corollary}

We first prove the pairing formula needed for (1).
\begin{lemma}\label{lem:pairing}
Suppose $n\ge3$, $h$ is balanced, and $\ii\partial\bp\omega=\Sigma$. For any real $(1,1)$-form $\alpha=\ii a_{p\bar q}\ph^p\wedge\bar\ph^q$,
\begin{equation}\label{eq:pairing}
 \ii\partial\bp\omega^{n-2}\wedge\alpha
 =\frac{1}{n(n-1)}\sum_{i<j}\sum_{p,q}
 a_{p\bar q}T^p_{ij}\overline{T^q_{ij}}\,\omega^n.
\end{equation}
\end{lemma}
\begin{proof}
Both sides are linear in $\alpha$ and invariant under unitary changes of frame. It suffices to take $\alpha=\ii\ph^n\wedge\bar\ph^n$. First assume $n\ge4$. Restrict to $W=\operatorname{span}_{\mathbb C}\{e_1,\ldots,e_{n-1}\}$, put $m=n-1$, $\beta=\omega|_W$, and $\theta^p=\tau^p|_W$. If $\zeta=(\partial\omega)|_W$, then
\[
 \zeta=\ii\sum_{p=1}^{m}\theta^p\wedge\bar\ph^p.
\]
The unitary linear-algebra formulas, for a $(2,0)$-form $\theta$ and a $(2,1)$-form $\psi$ on $W$, are
\begin{align}
 \theta\wedge\bar\theta\wedge\beta^{m-2}
 &=\frac{|\theta|^2}{m(m-1)}\beta^m,\label{eq:20pair}\\
 \ii\psi\wedge\bar\psi\wedge\beta^{m-3}
 &=\frac{|\Lambda\psi|^2-|\psi|^2}{m(m-1)(m-2)}\beta^m.\label{eq:21pair}
\end{align}
Here $\Lambda$ is the adjoint of exterior multiplication by $\beta$. To check \eqref{eq:21pair}, write $\psi=\psi_0+\beta\wedge\xi$, with $\Lambda\psi_0=0$. Then
\[
 |\psi|^2=|\psi_0|^2+(m-1)|\xi|^2,
 \qquad |\Lambda\psi|^2=(m-1)^2|\xi|^2.
\]
The two contributions to the wedge pairing are
\begin{align*}
 \ii\psi_0\wedge\bar\psi_0\wedge\beta^{m-3}
 &=-\frac{|\psi_0|^2}{m(m-1)(m-2)}\beta^m,\\
 \ii(\beta\wedge\xi)\wedge\overline{\beta\wedge\xi}\wedge\beta^{m-3}
 &=\ii\xi\wedge\bar\xi\wedge\beta^{m-1}
 =\frac{|\xi|^2}{m}\beta^m.
\end{align*}
The first sign can be checked on the primitive monomial $\ph^1\wedge\ph^2\wedge\bar\ph^3$; the coefficient is $(m-3)!/m!$. The mixed terms vanish because primitivity gives $\psi_0\wedge\beta^{m-2}=0$. Combining the two contributions with the displayed norm identities proves \eqref{eq:21pair}.

By \eqref{eq:power}, \eqref{eq:20pair}, and \eqref{eq:21pair},
\begin{align*}
 \left.\ii\partial\bp\omega^{n-2}\right|_W
 &=(m-1)\left(\sum_{p=1}^n\theta^p\wedge\bar\theta^p\wedge\beta^{m-2}
 +(m-2)\ii\zeta\wedge\bar\zeta\wedge\beta^{m-3}\right)\\
 &=\frac1m\left(\sum_{p=1}^n|\theta^p|^2-|\zeta|^2+|\Lambda\zeta|^2\right)\beta^m.
\end{align*}
Direct contraction yields
\[
 |\zeta|^2=\sum_{p=1}^m|\theta^p|^2,
 \qquad |\Lambda\zeta|^2=\sum_{i=1}^m\left|\sum_{k=1}^mT^k_{ki}\right|^2.
\]
Balancedness gives $\sum_{k=1}^mT^k_{ki}=-T^n_{ni}$. Thus the last displayed coefficient becomes
\[
 \frac1m\left(\sum_{i<j<n}|T^n_{ij}|^2+\sum_{i<n}|T^n_{ni}|^2\right)
 =\frac1m\sum_{i<j\le n}|T^n_{ij}|^2.
\]
Since $\beta^{n-1}\wedge\ii\ph^n\wedge\bar\ph^n=\omega^n/n$, this proves \eqref{eq:pairing}.

For $n=3$, wedge $\Sigma$ directly with $\ii\ph^3\wedge\bar\ph^3$. Its coefficient is $\sum_p|T^p_{12}|^2$. The trace condition gives $T^1_{12}=T^3_{23}$ and $T^2_{12}=-T^3_{13}$, so this coefficient is $\sum_{i<j}|T^3_{ij}|^2$. The factor is $1/6$, proving the same formula.
\end{proof}

\begin{proof}[Proof of Corollary~\ref{cor:special}]
By Theorem~\ref{thm:RBC}, $B_h=0$ and $R=0$. In particular $h$ is balanced and $\ii\partial\bp\omega=\Sigma$. It remains to show that $T=0$.

Suppose first that $\alpha$ is an astheno-K\"ahler form. Integration by parts gives
\[
 0=\int_M\omega\wedge\ii\partial\bp\alpha^{n-2}
 =\int_M\ii\partial\bp\omega\wedge\alpha^{n-2}
 =\int_M\sum_a\tau^a\wedge\bar\tau^a\wedge\alpha^{n-2}.
\]
For any $(2,0)$-form $\xi$,
\[
 \xi\wedge\bar\xi\wedge\alpha^{n-2}
 =\frac{|\xi|_\alpha^2}{n(n-1)}\alpha^n.
\]
The integrand is nonnegative, and vanishes precisely when $T=0$. This also proves the assertion for $n=2$, since $\alpha^{n-2}=1$.

If $n\ge3$ and $\alpha$ is pluriclosed, use Lemma~\ref{lem:pairing} instead:
\begin{align*}
 0&=\int_M\omega^{n-2}\wedge\ii\partial\bp\alpha
 =\int_M\ii\partial\bp\omega^{n-2}\wedge\alpha\\
 &=\frac{1}{n(n-1)}\int_M\sum_{i<j}
 \alpha\bigl(T(e_i,e_j),\overline{T(e_i,e_j)}\bigr)\,\omega^n.
\end{align*}
In the last line, $\alpha(\cdot,\overline{\cdot})$ denotes the associated Hermitian inner product. Its positivity forces $T=0$.

Thus $h$ is K\"ahler flat in either case. The finite torus cover follows from the structure theorem for compact flat K\"ahler manifolds.
\end{proof}

\begin{remark}
The K\"ahler conclusion in Corollary~\ref{cor:special} requires the additional hypothesis. Indeed, on the Iwasawa manifold there is a global holomorphic coframe with $d\ph^1=d\ph^2=0$ and $d\ph^3=\ph^1\wedge\ph^2$. The metric $\omega=\ii\sum_a\ph^a\wedge\bar\ph^a$ has zero Chern connection matrix in this frame, but $d\omega\ne0$.
\end{remark}

\section{Pluriclosed metrics with constant holomorphic sectional curvature}\label{sec:HSC}

\subsection{The curvature contraction}
We begin with the identity which replaces \eqref{eq:polar}.
\begin{lemma}\label{lem:HSCalgebra}
Let $h$ be a pluriclosed Hermitian metric with $H_h\equiv c$. Then
\begin{equation}\label{eq:HSCpair}
 R^h_{i\bar j k\bar\ell}+R^h_{k\bar\ell i\bar j}
 =c(h_{i\bar j}h_{k\bar\ell}+h_{i\bar\ell}h_{k\bar j})
 +\frac12h_{p\bar q}T^p_{ik}\overline{T^q_{j\ell}}.
\end{equation}
Consequently,
\begin{equation}\label{eq:scalarHSC}
 s_h=\frac{n(n+1)c}{2}+\frac14|T|_h^2.
\end{equation}
\end{lemma}
\begin{proof}
Polarization of $H_h\equiv c$ gives
\begin{align}
 R^h_{i\bar j k\bar\ell}+R^h_{k\bar j i\bar\ell}
 +R^h_{i\bar\ell k\bar j}+R^h_{k\bar\ell i\bar j}
 =2c(h_{i\bar j}h_{k\bar\ell}+h_{i\bar\ell}h_{k\bar j}).\label{eq:HSCpolar}
\end{align}
Indeed, let $\widehat R$ be one fourth of the left-hand side. It is symmetric in the holomorphic indices and in the antiholomorphic indices, and $\widehat R(v,\bar v,v,\bar v)=c|v|_h^4$. The tensor
\[
 \frac c2(h_{i\bar j}h_{k\bar\ell}+h_{i\bar\ell}h_{k\bar j})
\]
has the same symmetries and diagonal values. Comparing the coefficients of $v^iv^k\bar v^j\bar v^\ell$ gives \eqref{eq:HSCpolar}.

Since $h$ is pluriclosed, \eqref{eq:fourterm} becomes
\begin{equation}\label{eq:pluricurv}
 R^h_{i\bar j k\bar\ell}+R^h_{k\bar\ell i\bar j}
 -R^h_{k\bar j i\bar\ell}-R^h_{i\bar\ell k\bar j}
 =h_{p\bar q}T^p_{ik}\overline{T^q_{j\ell}}.
\end{equation}
This is the pluriclosed curvature identity in \cite[Lemma 2.1]{RZ}. Adding \eqref{eq:HSCpolar} and \eqref{eq:pluricurv} gives \eqref{eq:HSCpair}. Setting $j=i$, $\ell=k$ in an $h$-unitary frame and summing, we obtain \eqref{eq:scalarHSC}:
\[
 2s_h=c(n^2+n)+\frac12|T|_h^2.\qedhere
\]
\end{proof}

\subsection{Comparison with a K\"ahler--Einstein metric}
\begin{proof}[Proof of Theorem~\ref{thm:HSC}]
Since $X$ admits a K\"ahler metric, $h$ is pluriclosed, and $H_h<0$, the ampleness theorem of Broder--Stanfield \cite{BS} shows that $K_X$ is ample. By Aubin--Yau \cite{Aubin,Yau}, choose a K\"ahler--Einstein metric $g$ normalized by
\begin{equation}\label{eq:KEHSC}
 \Ric_g^{(1)}=a\omega_g,\qquad a=\frac{n+1}{2}c<0.
\end{equation}
The normalization follows by scaling a metric $g_0$ satisfying $\Ric_{g_0}^{(1)}=-\omega_{g_0}$ by $(-a)^{-1}$, since the first Ricci form is unchanged by constant scaling.

Put $u=\tr_g h$ and $v=\tr_h g$. At a point choose coordinates as in \eqref{eq:diag} and set
\[
 q=\sum_i\lambda_i^2,
 \qquad \mathcal T=\sum_{i,k}|T(\partial_i,\partial_k)|_h^2
 =\sum_{i,k,p}\lambda_p|T^p_{ik}|^2.
\]
The inputs in $\mathcal T$ are contracted by $g$, and the output by $h$. The full $h$-norm is
\[
 |T|_h^2=\sum_{i,k,p}\frac{\lambda_p}{\lambda_i\lambda_k}|T^p_{ik}|^2.
\]
By interchanging $i,k$, the two curvature terms in the contraction of \eqref{eq:HSCpair} have the same sum. Since
\[
 \sum_{i,k}h_{i\bar i}h_{k\bar k}=u^2,
 \qquad \sum_{i,k}h_{i\bar k}h_{k\bar i}=q,
\]
the contraction of \eqref{eq:HSCpair} gives
\begin{equation}\label{eq:QHSC}
 Q=\sum_{i,k}R^h_{i\bar i k\bar k}
 =\frac c2(u^2+q)+\frac14\mathcal T.
\end{equation}

We use the symmetric and skew-symmetric decomposition of the covariant Hessian, as in the Schwarz lemma argument of Broder--Stanfield \cite{BS}. Here the source metric $g$ is K\"ahler and the map is the identity. Let $A=\nabla\partial\mathrm{id}=\nabla^h-\nabla^g$ and let $S$ be its symmetric part in the two input indices. We use for $S$ the same induced norm as in Lemma~\ref{lem:CL}:
\[
 S^p_{ik}=\frac12(A^p_{ik}+A^p_{ki}).
\]
Because $g$ is torsion free, $A^p_{ik}-A^p_{ki}=T^p_{ik}$. Hence
\[
 A=S+\frac12T,\qquad
 |\nabla\partial\mathrm{id}|^2=|S|^2+\frac14\mathcal T.
\]
Indeed, in the coordinates \eqref{eq:diag},
\[
 \sum_{i,k,p}\lambda_p S^p_{ik}\overline{T^p_{ik}}
 =\sum_{i,k,p}\lambda_p S^p_{ki}\overline{T^p_{ki}}
 =-\sum_{i,k,p}\lambda_p S^p_{ik}\overline{T^p_{ik}},
\]
so $S$ and $T$ are orthogonal in this norm. Substituting into Lemma~\ref{lem:CL} and using \eqref{eq:QHSC}, the terms $\mathcal T/4$ cancel and we obtain
\begin{equation}\label{eq:HSCtraceexact}
 \Delta_g u=|S|^2+au-\frac c2(u^2+q).
\end{equation}
Since $q\ge u^2/n$ and $c<0$,
\begin{equation}\label{eq:HSCtracebound}
 \Delta_g u\ge\frac{(-c)(n+1)}{2n}\,u(u-n).
\end{equation}
Evaluating at a maximum point proves $u\le n$ throughout $X$. The inequality $uv\ge n^2$ then gives $v\ge n$.

Next, let $F=\log(\omega_g^n/\omega_h^n)$. By \eqref{eq:scalarHSC} and \eqref{eq:KEHSC},
\begin{align}
 \Delta_hF
 &=s_h-\tr_h\Ric_g^{(1)}\notag\\
 &=na+\frac14|T|_h^2-av\notag\\
 &=(-a)(v-n)+\frac14|T|_h^2\ge0.\label{eq:HSCvolume}
\end{align}
The strong maximum principle implies that $F$ is constant. Both nonnegative terms in \eqref{eq:HSCvolume} must vanish, so $T=0$ and $v=n$. From $uv\ge n^2$ and $u\le n$, we obtain $u=n$ and equality in Cauchy--Schwarz. Thus all $\lambda_i=1$, proving $h=g$.

In particular, $h$ is K\"ahler. Its curvature is therefore determined by $H_h=c$:
\[
 R^h_{i\bar j k\bar\ell}
 =\frac c2(h_{i\bar j}h_{k\bar\ell}+h_{i\bar\ell}h_{k\bar j}).
\]
Compactness implies completeness. The simply connected K\"ahler space form classification identifies the universal cover with $\CH^n$ of holomorphic sectional curvature $c$. This completes the proof.
\end{proof}

\begin{proof}[Proof of Corollary~\ref{cor:HSC}]
The negative case is Theorem~\ref{thm:HSC}. For $c=0$, the K\"ahler flatness of $h$ follows from \cite[Corollary 1.2]{BT}. The finite torus cover follows from the compact flat K\"ahler structure theorem.
\end{proof}

\begin{remark}
The argument for Theorem~\ref{thm:HSC} uses the pluriclosed condition in \eqref{eq:pluricurv} and in the canonical-bundle input, and uses the existence of a K\"ahler background to obtain $g$. For $c=0$, the positive coefficient in \eqref{eq:HSCtracebound} disappears, so the estimate $u\le n$ does not follow by setting $c=0$ in this proof. For a general Hermitian metric with $H_h=0$, polarization alone also fails to give \eqref{eq:pairzero}, the identity used in the real bisectional curvature argument. Thus the two proofs do not establish Conjecture~\ref{conj:HSC} without the stated additional assumption.
\end{remark}

\section*{Acknowledgements}
The author is grateful to Professor Fangyang Zheng for constant encouragement and support, and for introducing the author to problems concerning holomorphic sectional curvature and real bisectional curvature.

\section*{Declaration on the use of AI}
The author used ChatGPT Plus as auxiliary tools in this work. The author completed and verified all mathematical arguments and takes full responsibility for the content of this paper.

\end{document}